\documentclass[11pt,a4paper,leqno]{amsart}

\usepackage{amsmath,amssymb,amsfonts,amsthm,enumitem}
\setlist[enumerate]{label=(\roman*)}
\usepackage{hyperref}
\hypersetup{
	colorlinks = true,
	linkcolor = black,
	citecolor = black
}
\usepackage{tikz}
\usepackage{graphicx}
\usepackage{microtype}
\usetikzlibrary{arrows.meta}
\usepackage[T1]{fontenc}
\usepackage{lmodern}

\newtheorem{thm}{Theorem}[section]
\newtheorem{lemma}[thm]{Lemma}
\newtheorem{prop}[thm]{Proposition}
\newtheorem{cor}[thm]{Corollary}
\newtheorem{notation}[thm]{Notation}
\theoremstyle{definition}
\newtheorem{df}[thm]{Definition}

\theoremstyle{remark}
\newtheorem{rem}[thm]{Remark}
\numberwithin{equation}{section}
\theoremstyle{plain}
\newtheorem{question}{Question}
\newcounter{theoremintro}
\newtheorem{thmx}[theoremintro]{Theorem}
\newtheorem{corx}[theoremintro]{Corollary}

\DeclareMathOperator{\diag}{diag}

\newcommand{\Nb}{\mathbb{N}}
\newcommand{\Zb}{\mathbb{Z}}

\newcommand{\Cb}{\mathbb{C}}

\newcommand{\Rb}{\mathbb{R}}
\newcommand{\Zc}{\mathcal{Z}}

\newcommand{\cU}{\mathcal{U}}

\newcommand{\tr}{\mathrm{tr}}

\DeclareMathOperator{\Astar}{A*}

\allowdisplaybreaks

\begin{document}

\begin{abstract}
For a strictly ergodic action of a countable discrete group with Ozawa's property PHP on the Cantor set, we prove that dynamical comparison implies complete selflessness of its reduced crossed product. Therefore these crossed products have stable rank one, strict comparison, a unique quasitracial state, and real rank zero. Applications include probability-measure-preserving Cantor actions of C$^*$-simple acylindrically hyperbolic and C$^*$-simple linear groups for which an infinite amenable subgroup acts strictly ergodically. The results also apply to subodometers and infinite regular Toeplitz subshifts of residually finite groups with PHP.
\end{abstract}


 
\title{Selfless nonnuclear crossed products}
\author{Jamie Bell}
\address[Jamie Bell]{Mathematical Institute, University of M\"unster, Einsteinstrasse 62, 48149 M\"unster, Germany.}
\email{jbell@uni-muenster.de}

\thanks{Funded by the Deutsche Forschungsgemeinschaft (DFG, German Research Foundation) under Germany's Excellence Strategy EXC 2044/2 –390685587, Mathematics M\"unster: Dynamics-Geometry-Structure and project-ID 427320536, SFB 1442, of the DFG}

\maketitle

\section{Introduction}

Three fundamental properties of a $\mathrm{II}_1$ factor are that (i) Murray--von Neumann comparison of projections is determined by its unique trace, (ii) every element may be approximated by an invertible one (using unitary polar decomposition), and (iii) every self-adjoint element may be approximated by an invertible self-adjoint element (by Borel functional calculus). In the realm of unital simple finite C$^*$-algebras, none of these three properties is automatic and understanding the conditions under which they hold is a central problem in the structure theory of C$^*$-algebras. 

The approximation properties in (ii) and (iii) are precisely \emph{stable rank one} and \emph{real rank zero}, respectively, as introduced by Rieffel \cite{Rie83} and Brown and Pedersen \cite{BroPed91}. The analogue of (i), known as \emph{strict comparison}, requires replacing comparison of projections by Cuntz comparison of positive elements. For unital simple nuclear C$^*$-algebras, these properties are particularly well-understood. Strict comparison is tightly related to the central regularity hypothesis of $\Zc$-stability which undergirds the many successes of the Elliott classification programme. Indeed, the Toms--Winter conjecture predicts that these two properties are equivalent among nonelementary simple separable nuclear C$^*$-algebras. R{\o}rdam \cite{Ror04} proved one implication, showing that for unital simple separable finite C$^*$-algebras, $\Zc$-stability implies both strict comparison and stable rank one, along with a simple criterion for establishing real rank zero.  

Over the past couple of decades, significant effort has been expended addressing the question of when $\Zc$-stability holds for unital simple nuclear C$^*$-algebras. One natural setting is that of crossed products arising from topological group actions. Kerr's notion of \emph{almost finiteness} \cite{Ker20}, modelled after an earlier notion introduced by Matui \cite{Mat12}, leads to the most definitive results, implying $\Zc$-stability for free minimal actions of infinite (amenable) groups. For free actions of countably infinite amenable groups on the Cantor set, almost finiteness is equivalent to \emph{(dynamical) comparison}, which may be thought of as a dynamical analogue of strict comparison \cite{KerSza20}. Roughly speaking, it asks that if one set is smaller than another with respect to all invariant measures on the space, then the first may be disassembled and translated disjointly via group elements into the other set. For actions of amenable groups on the Cantor set, comparison is now known to be automatic \cite{GlaLiu26} (see also \cite{Ker20,KerSza20,DowZha23,Nar22,Nar24,NarPet25,KerNar25}). For free minimal Cantor actions of countably infinite amenable groups, it follows that their crossed products have strict comparison, stable rank one, and real rank zero \cite{AraBonBosLi20}.  

Given the definitive results we have in the nuclear regime, it is natural to wonder which of these properties continue to be present for unital simple nonnuclear C$^*$-algebras. A first step was made by Dykema, Haagerup, and R{\o}rdam \cite{DykHaaRor97} who proved that many reduced free products, including the reduced C$^*$-algebras of free groups $C^*_\lambda(F_d)$, have stable rank one. They are moreover simple by Powers' classical result \cite{Pow75}. Their methodology, which relied on a rapid-decay-type estimate to control the operator norm in terms of the $2$-norm, was successfully pushed to encompass a whole host of other C$^*$-simple groups \cite{DykHar99,GerOsi20}. Pimsner and Voiculescu famously proved that $C^*_\lambda(F_d)$ has no nontrivial projections \cite{PimVoi82}, and so is very far from having real rank zero. However, the question of strict comparison for $C^*_\lambda(F_d)$ remained a longstanding open problem until its recent positive resolution in the work of Amrutam, Gao, Kunnawalkam Elayavalli, and Patchell \cite{AmrGaoKunPat25}. 

Their proof establishes \emph{selflessness}, a property introduced by Robert \cite{Rob25}. This is not unrelated to the work of Dykema, Haagerup, and R{\o}rdam on reduced free products. Indeed, a tracial C$^*$-algebra $(A,\tau)$ is selfless precisely when the diagonal embedding into its norm ultrapower can be factored in a precise way through a nontrivial reduced free product $(A,\tau) * (C,\psi)$. This allows one to transfer regularity properties of infinite reduced free products to $A$ itself. Robert used this to prove that, in the presence of a trace, selflessness implies simplicity, stable rank one, strict comparison and uniqueness of the quasitracial state. More recently, he established a converse: infinite-dimensional simple unital monotracial C$^*$-algebras with strict comparison with respect to their trace are selfless \cite{Rob26}. For this reason, one might profitably think of selflessness (or the closely related notion of pureness) as a veritable nonnuclear analogue of $\Zc$-stability -- at least in the simple monotracial setting. 

These developments have generated considerable activity in recent years, with selflessness being established for broad classes of reduced free products and graph products \cite{HayKunRob25,FloKliCobPag26}, amalgamated free products and HNN extensions \cite{GaoKunPatTer26}. Influential work of Ozawa \cite{Oza25} supplied further methods for establishing selflessness which remove the rapid decay hypotheses used in earlier works. This has led to several other applications; see  \cite{HayKunPatRob25,FloKliCobPag26b,LeVLeiVig26,MaWanYan25,Vig25,Vig26}. 
  
Despite the rapid rate of progress, extending this picture from reduced group C$^*$-algebras to crossed products has proved particularly recalcitrant. A natural starting point is to consider strictly ergodic, i.e.\ minimal and uniquely ergodic, actions of free groups on the Cantor set. Their crossed products, while nonnuclear, are simple finite and monotracial \cite{BryKen18}, so fall cleanly within the selflessness framework. Moreover, for a free product $G=G_1*G_2$, there is a natural decomposition of the reduced crossed product
\[
 C(X)\rtimes_\lambda G  \cong (C(X)\rtimes_\lambda G_1)*_{C(X)}(C(X)\rtimes_\lambda G_2)
\]
as an amalgamated free product over $C(X)$. Despite this, the available criteria for amalgamated free products do not directly apply. Gao, Junge, Kunnawalkam Elayavalli, Patchell, and Robert \cite{GaoJunKunPatRob26} establish selflessness of canonical conditional expectations for approximately inner actions of a large class of nonamenable groups, and obtain crossed product selflessness under additional hypotheses, while Ohshima \cite{Ohs26} treats certain almost periodic actions on simple exact $\mathcal{Z}$-stable monotracial C$^*$-algebras. However an action on a commutative C$^*$-algebra can never be approximately inner (unless it is trivial), and the only simple commutative C$^*$-algebra is $\Cb$. Hence, neither of these address the question of selflessness for actions on the Cantor set. 

The first progress on this question was made by the author in \cite{Bel26} where, using the special inductive limit structure of crossed products of odometer actions, selflessness could be transferred from group C$^*$-algebras to nonnuclear Bunce--Deddens algebras. Crossed products of commutative C$^*$-algebras have also been studied by Flores, Klisse, {\'O} Cobhthaigh, and Pagliero \cite{FloKliCobPag26b}, who introduce an action-dependent version of Ozawa’s PHP property and use it to establish selfless inclusions of reduced group C$^*$-algebras into crossed products. Their applications include purely infinite crossed products arising from boundary and projective actions, as certain tracial examples. The latter are obtained using a nontrivial-kernel criterion \cite[Theorem 4.4]{FloKliCobPag26b}, so in particular the actions are very far from being free.  

These results provide important examples of selfless crossed products. Establishing regularity for broad classes of strictly ergodic Cantor actions, however, calls for a new approach. One such approach is explored in the author's joint work with Geffen and Kerr \cite{BelGefKer25} on stable rank one in simple nonnuclear crossed products. For a generic action in each of two natural classes of minimal $F_d$-actions on the Cantor set, the reduced crossed product has stable rank one. The key insight was to exploit methods from amenable dynamics, such as F{\o}lner tilings and Rokhlin tower decompositions, to extend strong local forms of regularity to weaker properties at the global scale. This philosophy was pushed further in recent work on stable rank one for certain C$^*$-simple topological full groups \cite{KerPet26}. 

The genericity results in \cite{BelGefKer25} rely on Baire category arguments, and it is unclear how these methods might yield further regularity properties such as strict comparison. The present paper brings this discussion full circle by connecting dynamical comparison with selflessness. We show that dynamical comparison, already a central tool for establishing regularity in the nuclear setting, also yields strong regularity properties for simple nonnuclear crossed products. Our main result is the following. 

\begin{thmx}\label{T:PHP-main}
Let $G$ be a countable discrete group with Ozawa's property PHP, and let $G\curvearrowright X$ be a strictly ergodic action on the Cantor set with dynamical comparison. Then $C(X)\rtimes_\lambda G$ is completely selfless. Consequently, the crossed product has stable rank one, strict comparison, a unique quasitracial state, and real rank zero.
\end{thmx}

Property PHP, an initialism for Powers--Haagerup--Pisier, was recently introduced by Ozawa in \cite{Oza25}, who proved that it implies complete selflessness of its corresponding reduced group C$^*$-algebra. Thus, Theorem~\ref{T:PHP-main} may be seen as a crossed product generalisation of Ozawa's result. Property PHP is enjoyed by a large class of discrete groups including C$^*$-simple linear groups \cite{Vig26}, C$^*$-simple acylindrically hyperbolic groups and nonamenable free products \cite[Proposition 15]{Oza25}. In particular, we have the following. 

\begin{corx}\label{T:main}
Let $G$ be a nonamenable free product of two nontrivial groups, and let $G\curvearrowright X$ be a strictly ergodic action on the Cantor set with dynamical comparison. Then $C(X)\rtimes_\lambda G$ is completely selfless.
\end{corx}

The proof of Theorem~\ref{T:PHP-main} combines Ozawa's PHP methodology with a key new dynamical averaging property that is accessed via (dynamical) comparison. Indeed, we start by reformulating comparison for strictly ergodic actions on the Cantor set in terms of what we call \emph{balanced tower partitions} (Proposition~\ref{P:comparison}). This condition asks that there is a partition of the Cantor set $X$ into clopen towers $(B,S)$ of a prescribed minimum height such that, on some prescribed finite collection of functions in $C(X)$, we have 
\[
 \frac1{|S|}\sum_{s\in S}f(sy)\approx \int_X f\,d\mu , \quad y\in B
\]
uniformly over all the towers and the prescribed functions. The shapes $S$ are arbitrary subsets of the group and, in particular, do not impose amenability of the acting group (unlike almost finiteness and the related uniform Rokhlin property). 

Writing $A=C(X)\rtimes_\lambda G$ for the reduced crossed product, $E:A\to C(X)$ for the canonical conditional expectation, and $\tau=\mu\circ E$, we use these balanced tower partitions to obtain the following \emph{norm averaging} property, which we expect to be of independent interest (Corollary~\ref{C:comparison-NA}): 
\begin{equation*}\tag{NA}\label{E:NA-intro}
 \forall F\subset A\text{ finite}\ \forall\varepsilon>0\quad
 \exists w\in\cU(A)\quad
 \max_{x\in F}\|E(w^*xw)-\tau(x)1\|<\varepsilon.
\end{equation*}

Norm averaging makes the identity Fourier coefficients of a prescribed finite collection approximately scalar after conjugation by a single unitary. Ozawa's PHP construction then controls the remaining Fourier coefficients. The averaging unitary is constructed independently of the PHP data, separating the dynamical part of the argument from the group-theoretic part.

Theorem~\ref{T:PHP-main} is broadly applicable. As noted earlier in the introduction, actions of infinite amenable groups on the Cantor set automatically have comparison. Both strict ergodicity and comparison pass from a subgroup action to the full action, provided the latter preserves a probability measure (we record this observation as Proposition~\ref{P:subgroup-comparison}). We therefore obtain the following consequence.

\begin{corx}\label{T:amenable}
Let $G$ be a countable discrete group with property PHP. Suppose that $G\curvearrowright X$ preserves a Borel probability measure on the Cantor set and that some infinite amenable subgroup acts strictly ergodically on $X$. Then the reduced crossed product is completely selfless.
\end{corx}

Specialising to Cantor actions of nonamenable free groups, we obtain a strengthening of \cite[Theorem D]{BelGefKer25}. In what follows, $\Astar(F_d,X)$ denotes the space of topologically free actions preserving a Borel probability measure whose standard generators act as strictly ergodic and spectrally aperiodic homeomorphisms, introduced in \cite{BelGefKer25}.  

\begin{corx}\label{T:free}
Let $d\ge 2$, and suppose that an action $F_d\curvearrowright X$ on the Cantor set preserves a Borel probability measure. If some nonidentity element of $F_d$ acts strictly ergodically, then $C(X)\rtimes_\lambda F_d$ is completely selfless, has stable rank one, strict comparison, a unique quasitracial state, and real rank zero. In particular, these conclusions hold for every action in $\Astar(F_d,X)$.
\end{corx}

Since subodometer actions of residually finite groups are known to have comparison, we additionally recover the selflessness conclusions in \cite{Bel26}.

\begin{corx}\label{C:BD}
	Let $G$ be a countably infinite residually finite group with property PHP, and let $\sigma$ be a separating normal chain in $G$. Then the Bunce--Deddens algebra $\mathsf{BD}(G,\sigma)$ is completely selfless, has stable rank one, strict comparison, a unique quasitracial state, and real rank zero. 
\end{corx}

Further examples of actions to which Theorem~\ref{T:PHP-main} may be applied are presented in Section~\ref{S:Toeplitz} using recent work of Cecchi Bernales, Cortez, and G{\'o}mez. These examples extend the applications beyond the equicontinuous setting of Corollary~\ref{C:BD}. Moreover, they can be chosen so that no amenable subgroup acts minimally, and hence lie outside the scope of Corollary~\ref{T:amenable}; see Remark 7.3.

\begin{corx}\label{C:Toeplitz}
	Let $G$ be a countably infinite residually finite group with property PHP, and let $G\curvearrowright X$ be an infinite regular Toeplitz subshift. Then $C(X)\rtimes_\lambda G$ is completely selfless.
\end{corx} 

The paper is structured as follows. In Section~\ref{S:prelim}, we collect the relevant preliminaries. Section~\ref{S:towers} characterises comparison by balanced towers, and Section~\ref{S:average} establishes norm averaging. Section~\ref{S:group-PHP} combines norm averaging with PHP and in 
Section~\ref{S:applications} we deduce our main results. Section~\ref{S:Toeplitz} constructs further examples of actions covered by Theorem~\ref{T:PHP-main}. 

\bigskip\noindent\textit{Acknowledgements}.
I thank the organisers and participants of the workshop ``Cuntz
Semigroups and Dynamics'' for fostering a stimulating environment in which many ideas related to this work were discussed. In particular, I thank Leonel Robert for his inspiring talk on selflessness. 

\bigskip\noindent\textit{AI declaration}.
ChatGPT-6.0 Astra Pro was used during the preparation of this work, including identifying an earlier version of the averaging argument, creating a draft of the manuscript, proofreading, and creating the figures. The author subsequently verified and simplified the mathematical arguments and takes full responsibility for the contents of the paper.

\section{Preliminaries}\label{S:prelim}

\begin{notation}
Throughout, $X$ denotes the Cantor set, $G$ a countable discrete
group with identity $e$, and $\Nb=\{1,2,\ldots\}$.
\end{notation}

\subsection{Cantor actions and comparison}
An action $G\curvearrowright X$ is \emph{strictly ergodic}
if it is minimal and uniquely ergodic. We write $M_G(X)$ for the
invariant Borel probability measures. An invariant measure of a
minimal Cantor action has full support and no atoms. An action of $\Zb$ will be identified with the homeomorphism corresponding to $1\in \Zb$. We also write $T\curvearrowright X$ for a homeomorphism $T : X \to X$. 

For clopen subsets $U,V\subseteq X$, write $U\prec_G V$ and say $U$ is \emph{(dynamically) subequivalent} to $V$, if there is a finite clopen partition $U=\bigsqcup_j U_j$ and elements
$g_j\in G$ such that the sets $g_jU_j$ are pairwise disjoint
subsets of $V$. 

\begin{df}
An action $G\curvearrowright X$ has \emph{dynamical
comparison} if $U\prec_G V$ whenever $U,V\subseteq X$ are nonempty clopen and
$\mu(U)<\mu(V)$ for every $\mu\in M_G(X)$.
\end{df}

For Cantor actions, this is equivalent to the usual formulation
with a closed source and an open target (see \cite[Proposition 3.6]{Ker20}). 

\begin{prop}\label{P:subgroup-comparison}
Suppose that $G\curvearrowright X$ preserves a probability measure
and that a subgroup $H\le G$ acts strictly ergodically with
dynamical comparison. Then the $G$-action is strictly ergodic and
has dynamical comparison.
\end{prop}

\begin{proof}
Since $H \le G$, we have $Hx \subseteq Gx$ for every $x\in X$ and so since $H\curvearrowright X$ is strictly ergodic (in particular minimal) $X = \overline{Hx}\subseteq \overline{Gx} \subseteq X$, thus $G\curvearrowright X$ is minimal. Since every $G$-invariant probability measure is $H$-invariant, the measure $\mu$ preserved by the $G$-action belongs to $M_H(X)$. By unique ergodicity of $H\curvearrowright X$, we conclude $\{\mu\} \subseteq M_G(X) \subseteq M_H(X) = \{\mu\}$, which shows that $G\curvearrowright X$ is uniquely ergodic and hence strictly ergodic. Finally, let $U,V\subseteq X$ be clopen and suppose that $\mu(U) < \mu(V)$. Then $U\prec_H V$ via comparison for $H\curvearrowright X$ and in particular $U\prec_G V$. 
\end{proof}

\subsection{Crossed products and their expectations}
Let a group $G$ act on a unital C$^*$-algebra $D$ by
automorphisms $\alpha_g$. We write $A=D\rtimes_\lambda G$ for
the reduced crossed product. This is a certain completion of the algebraic crossed product $D\rtimes_{\mathrm{alg}} G$, which is the $*$-algebra of Fourier polynomials $\sum_{g\in F} d_g u_g$, where $F\subseteq G$ is finite, $u_g$ are fixed unitaries indexed by $g\in G$ and $d_g \in D$, subject to the relations $u_gu_h = u_{gh}$ and $u_gdu_g^*=\alpha_gd$. The reduced completion is distinguished by the existence of a faithful conditional expectation $E:A\to D$, which on $D\rtimes_{\mathrm{alg}} G$ is given by
\[
 E\Big(\sum_{g\in F}d_gu_g\Big)=d_e.
\]
An invariant faithful trace $\sigma$ on $D$ gives a faithful trace $\tau=\sigma\circ E$ on $A$. For a Cantor action, $\alpha_gf(x)=f(g^{-1}x)$ and we write $\mu(f)=\int f\,d\mu$. We refer the reader to \cite{BroOza08} for further background.

\begin{df}\label{D:NA}
Let $E:A\to D$ be a faithful conditional expectation and let
$\tau=\sigma\circ E$ be a faithful tracial state. We say that
$(E,\tau)$ has \emph{norm averaging} if, for every finite
$F\subset A$ and $\varepsilon>0$, there is a unitary $w\in A$
such that, for all $x\in F$, 
\begin{equation}\label{E:NA}
 \|E(w^*xw)-\tau(x)1\|<\varepsilon.
\end{equation}
\end{df}

\subsection{Selflessness and complete selflessness}
A C$^*$-probability space is a unital C$^*$-algebra with a state,
assumed to have faithful GNS representation unless stated otherwise.
All ultrapowers below are norm ultrapowers.

\begin{df}[\cite{Rob25}]
A unital C$^*$-algebra $A$ with faithful tracial state $\tau$ is
\emph{selfless} if there exist a nontrivial C$^*$-probability space
$(Q,\psi)$, a free ultrafilter $\cU$, and a state-preserving embedding
\[
 (A,\tau)*(Q,\psi)\longrightarrow (A^{\cU},\tau^{\cU})
\]
whose restriction to $A$ is the diagonal embedding.
\end{df}

By \cite[Theorem 3.1]{Rob25}, a selfless tracial C$^*$-algebra
$(A,\tau)$ is simple, has stable rank one and strict comparison,
and has $\tau$ as its unique quasitracial state.

Let $(A,\tau)\subseteq(B,\rho)$ be an inclusion preserving the
states. It is \emph{completely existential} if there is a
state-preserving embedding $B\to A^{\cU}$ extending the diagonal
inclusion such that, for every C$^*$-algebra $C$, the induced map
on the algebraic tensor product extends continuously to
\[
 B\otimes_{\min}C\longrightarrow(A\otimes_{\min}C)^{\cU}.
\]
The pair $(A,\tau)$ is \emph{completely selfless} if its inclusion
in $(A,\tau)*(Q,\psi)$ is completely existential for some nontrivial
$(Q,\psi)$; see \cite[Section 6]{Oza25}.

The following is the criterion we will apply.

\begin{thm}[Ozawa, {\cite[Theorem 13]{Oza25}}]\label{T:Ozawa-isometry}
Suppose that $(A,\tau)\subseteq(B,\psi)$ is a unital
state-preserving inclusion and that there is an isometry
$V\in B^{\cU}$ satisfying, for all $x\in A$,
\[
 V^*xV=\tau(x)1,\quad
 \psi^{\cU}(xVV^*x^*)=0,\quad
 V+V^*\in A^{\cU}.
\]
Then $(A,\tau)$ is completely selfless.
\end{thm}

When referring to selflessness of a crossed product, we sometimes suppress the trace. It will be understood that we always mean the canonical trace induced by the invariant probability measure for the action.

\subsection{Ozawa's property PHP}
\begin{df}\label{D:group-PHP}
A group $G$ has \emph{property PHP} if, for every finite
$F\subseteq G$ and $\eta>0$, there is $N\in\Nb$ such that
for every $n\ge N$ there are $t_i\in G$ and subsets
$C_i\subseteq D_i\subseteq G$, $1\le i\le n$, for which
the indexed family
\[
 \{aC_i:a\in F,\ 1\le i\le n\}
 \ \cup\
 \{bt_i^{-1}(G\setminus D_i):b\in F,\ 1\le i\le n\}
\]
is pairwise disjoint, and
\begin{equation}\label{E:group-PHP-overlap}
 \Big\|\sum_{i=1}^n1_{D_i}
       +\sum_{i=1}^n1_{t_i^{-1}(G\setminus C_i)}
 \Big\|_\infty\le\eta\sqrt n.
\end{equation}
\end{df}

\section{Comparison and balanced clopen towers}\label{S:towers}

A \emph{clopen tower} is a pair $(B,S)$, where $B\subseteq X$ is
nonempty and clopen, $S\subseteq G$ is finite and nonempty, and
the levels $sB$, $s\in S$, are pairwise disjoint. A \emph{tower
partition} is a finite family of such towers whose levels partition
$X$. We do not require $e\in S$.

\begin{df}\label{D:balanced}
Let $\mu$ be an invariant probability measure. The action has
\emph{balanced towers with respect to $\mu$} if, for every finite
$\mathcal F\subset C(X)$, $\delta>0$, and $N\in\Nb$, there is a
clopen tower partition $\{(B_i,S_i)\}_i$ with $|S_i|\ge N$ and for all $f\in \mathcal{F}$ and $y\in B_i$,
\begin{equation}\label{E:balanced}
 \Big|\frac1{|S_i|}\sum_{s\in S_i}f(sy)-\mu(f)\Big|<\delta. 
\end{equation}
\end{df}

\begin{prop}\label{P:comparison}
A strictly ergodic action $G\curvearrowright X$ of a countable discrete group on the Cantor set has dynamical comparison if and only if it has balanced towers with respect to its unique invariant measure.
\end{prop}

The construction can be viewed as packing and stacking; see
Figure~\ref{fig:balanced-towers}. Colour $X$ by a finite clopen
partition $P_1,\ldots,P_k$. Comparison allows us to fill each
$P_i$ with disjoint translates of a small clopen
set $B$, leaving between one and two copies' worth of measure
unfilled. The numbers of copies are chosen in approximately the
proportions $\mu(P_i)$. Refining their common source $B$ turns
them into genuine tower levels.

To obtain a partition of all of $X$, we must also include the
remainder $R$. Its measure is less than that of $2k$ copies of
$B$, so comparison embeds $R$ into $2k$ of the copies already
placed. Pulling back through this embedding gives at most $2k$ partial piecewise group translations from clopen subsets of $B$ into $R$. A further refinement makes each partial map either undefined or a single translation on every new base. Thus each tower acquires at most $2k$ additional levels,
while its original height can be arbitrarily large.

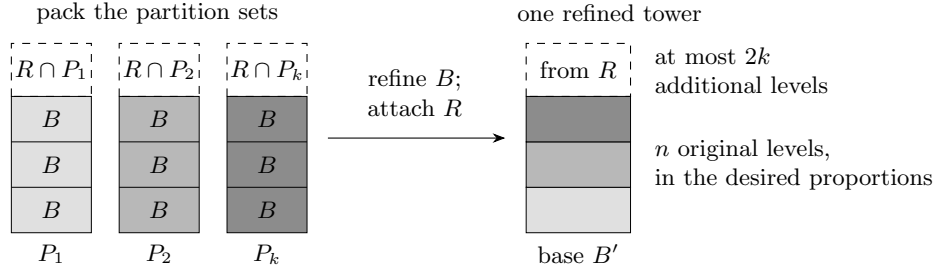
\begin{figure}[htbp]
\centering
\begin{tikzpicture}[x=0.92cm,y=1cm,font=\footnotesize,>=Stealth]
 \foreach \x/\shade/\lab in {0/12/1,1.55/28/2,3.1/45/k} {
   \filldraw[fill=black!\shade,line width=0.35pt]
     (\x,0) rectangle +(1.15,1.8);
   \draw[line width=0.3pt] (\x,0.6) -- +(1.15,0);
   \draw[line width=0.3pt] (\x,1.2) -- +(1.15,0);
   \foreach \y in {0.3,0.9,1.5}
     \node at (\x+0.575,\y) {$B$};
   \draw[dashed,line width=0.4pt] (\x,1.8) rectangle +(1.15,0.7);
   \node at (\x+0.575,2.15) {$R\cap P_{\lab}$};
   \node at (\x+0.575,-0.28) {$P_{\lab}$};
 }
 \node at (2.1,2.9) {pack the partition sets};
 \draw[->] (4.55,1.25) -- (7.1,1.25);
 \node[align=center] at (5.8,1.85) {refine $B$;\\attach $R$};
 \foreach \j/\shade in {0/12,1/28,2/45} {
   \filldraw[fill=black!\shade,line width=0.35pt]
     (7.4,0.6*\j) rectangle +(1.5,0.6);
 }
 \draw[dashed,line width=0.4pt] (7.4,1.8) rectangle +(1.5,0.7);
 \node at (8.15,2.15) {from $R$};
 \node at (8.15,-0.28) {base $B'$};
 \node at (8.6,2.9) {one refined tower};
 \node[anchor=west,align=left] at (9.1,0.9)
   {$n$ original levels,\\in the desired proportions};
 \node[anchor=west,align=left] at (9.1,2.15)
   {at most $2k$\\additional levels};
\end{tikzpicture}
\caption{Schematic construction of balanced towers.}
\label{fig:balanced-towers}
\end{figure}

\begin{proof}
$[\Rightarrow]$. We first prove the following. 

\textbf{Claim.} Given a partition $X=P_1\sqcup\cdots\sqcup P_k$ into nonempty clopen sets,
$\eta>0$, and $N\in\Nb$, there is a tower partition with heights
at least $N$ such that every level is contained in one of the
$P_i$ and, in every tower, the proportion of levels in $P_i$
differs from $p_i:=\mu(P_i)$ by less than $\eta$.

\smallskip\noindent\textit{Step 1.}
Choose a nonempty clopen $B$ of measure $b=\mu(B)>0$, and set
\[
 n_i=\Big\lfloor\frac{p_i}{b}\Big\rfloor-1,
 \qquad n=\sum_{i=1}^k n_i.
\]
The measure $b$ can be arbitrarily small since $\mu$ is regular and nonatomic. Since all $p_i$ are positive and $n>1/b-2k$, we can choose $B$ so that all $n_i$
are positive and
\begin{equation}\label{E:height-choice}
 n\ge\max\{N,2k\},\qquad \frac{4k+2}{n}<\eta.
\end{equation}
We also have, for all $1\le i\le k$,
\begin{equation}\label{E:remainder-measure}
 b\le p_i-n_i b<2b.
\end{equation}

For each $i$, we successively place $n_i$ disjoint copies of
$B$ in $P_i$ using comparison. Indeed, after
$\ell<n_i$ copies have been placed, the unused part of $P_i$
is clopen and has measure
\[
 p_i-\ell b=(p_i-n_i b)+(n_i-\ell)b\ge2b>b.
\]
Comparison therefore supplies the next copy. Each image has
measure $b$, since subequivalences preserve $\mu$.
Enumerate the resulting maps by
$\varphi_1,\ldots,\varphi_n:B\to X$, and put
$Y_j=\varphi_j(B)$. The $Y_j$ are pairwise disjoint clopen
sets, exactly $n_i$ of which lie in $P_i$.

\smallskip\noindent\textit{Step 2.}
Let $R=X\setminus\bigsqcup_{j=1}^nY_j$. Summing
\eqref{E:remainder-measure} gives
\[
 kb\le\mu(R)<2kb.
\]
Since $n\ge2k$, choose $2k$ of the ranges and write
$V=\bigsqcup_{j=1}^{2k}Y_j$. Then $\mu(R)<\mu(V)$, so
comparison gives a subequivalence $\theta:R\to V$. For $1\le j\le2k$, define the clopen subset
\[
 D_j=\varphi_j^{-1}(Y_j\cap\theta(R))\subseteq B
\]
and maps 
\[
 \psi_j=\theta^{-1}\circ\varphi_j|_{D_j}:D_j\rightarrow R.
\]
These maps have pairwise disjoint images, and
\begin{equation}\label{E:remainder-partition}
 R=\bigsqcup_{j=1}^{2k}\psi_j(D_j).
\end{equation}
Indeed, the sets $Y_j\cap\theta(R)$ partition $\theta(R)$,
and applying $\theta^{-1}$ gives exactly this identity.
Some $D_j$ may be empty.

\smallskip\noindent\textit{Step 3.}
Take a finite partition $\mathcal B$ of $B$ into nonempty clopen
sets such that, for each $B'\in\mathcal B$:
\begin{enumerate}
\item every $\varphi_j|_{B'}$ is a single group translation;
\item for each $j\le2k$, either $B'\subseteq D_j$ or
      $B'\cap D_j=\emptyset$;
\item when $B'\subseteq D_j$, the map $\psi_j|_{B'}$ is a
      single group translation and $\psi_j(B')$ lies in one
      member of $\{P_1,\ldots,P_k\}$.
\end{enumerate}
Such a partition exists by taking a common refinement of the
finitely many clopen pieces defining the $\varphi_j$ and
$\psi_j$, the sets $D_j$, and the clopen sets
$\psi_j^{-1}(P_i)$.

For $B'\in\mathcal B$, write
$\varphi_j|_{B'}=g_{j,B'}|_{B'}$, and, whenever
$B'\subseteq D_j$, write $\psi_j|_{B'}=h_{j,B'}|_{B'}$.
Define
\[
 S_{B'}=\{g_{j,B'}:1\le j\le n\}
       \cup\{h_{j,B'}:1\le j\le2k,\ B'\subseteq D_j\}.
\]
The corresponding images of $B'$ are disjoint. The original
images lie in distinct $Y_j$, the additional images lie in the
disjoint sets in \eqref{E:remainder-partition}, and no original
image meets $R$. In particular, all the displayed group elements
are distinct since equal elements would give equal nonempty images of
$B'$. Thus $(B',S_{B'})$ is a clopen tower of height
\[
 |S_{B'}|=n+t_{B'},\qquad
 t_{B'}=|\{j\le2k:B'\subseteq D_j\}|\le2k.
\]

These towers form a partition of $X$. For fixed $j$, the sets
$\varphi_j(B')$, as $B'$ runs through $\mathcal B$, partition
$Y_j$. Likewise, the sets $\psi_j(B')$ with $B'\subseteq D_j$
partition $\psi_j(D_j)$. Together with
\eqref{E:remainder-partition}, this proves both coverage and
disjointness of all levels, including levels from different towers.

\smallskip\noindent\textit{Step 4.}
Fix a resulting tower and abbreviate $t=t_{B'}$. Let $r_i$ be
the number of its additional levels lying in $P_i$. Its total
number of levels in $P_i$ is $n_i+r_i$, where
$0\le r_i\le t$ and $\sum_i r_i=t$. Put
\[
 d_i= \frac{p_i}{b}-n_i\in[1,2),\qquad
 d=\sum_{i=1}^k d_i=\frac1b - n \in[k,2k).
\]
Since $n_i-p_i n=p_i d-d_i$, we have
\[
 \left|\frac{n_i}{n}-p_i\right|<\frac{2k+2}{n}.
\]
Adding the extra levels changes this proportion by at most
\[
 \left|\frac{n_i+r_i}{n+t}-\frac{n_i}{n}\right|
 =\frac{|nr_i-n_i t|}{n(n+t)}
 \le\frac{t}{n+t}\le\frac{2k}{n}.
\]
Consequently,
\begin{equation}\label{E:balance-bound}
 \left|\frac{n_i+r_i}{n+t}-p_i\right|
 <\frac{4k+2}{n}<\eta.
\end{equation}
Every height is at least $n\ge N$, proving the claim.

We now deduce balanced towers for continuous functions. Given
finite $\mathcal F\subset C(X)$ and $\delta>0$, choose a common
clopen partition $P_1,\ldots,P_k$ and functions
$f_0=\sum_i c_{f,i}1_{P_i}$ with
$\|f-f_0\|<\delta/3$ for $f\in\mathcal F$.
This is possible because locally constant functions are uniformly
dense in $C(X)$. Choose $\eta>0$ so that
$\eta\sum_i|c_{f,i}|<\delta/3$ for every $f\in\mathcal F$,
and apply the claim with the prescribed minimum height $N$. Let $q_i$ denote the proportion of levels in $P_i$ in a tower $(B',S)$. The tower average of $f_0$ is $\sum_i c_{f,i}q_i$, while $\mu(f_0)=\sum_i c_{f,i}p_i$. Approximating $f$ by $f_0$ in both the tower average and the integral therefore gives, for every $y\in B'$,
\begin{align*}
 \Bigg|\frac1{|S|}\sum_{s\in S}f(sy)-\mu(f)\Bigg| &\le2\|f-f_0\|+\Bigg|\sum_i c_{f,i}(q_i-p_i)\Bigg|\\
 &\le2\|f-f_0\|+\sum_i|c_{f,i}|\,|q_i-p_i|<\delta.
\end{align*}
This proves balanced towers.

\noindent $[\Leftarrow]$. Let $U,V\subseteq X$ be clopen with $\mu(U)<\mu(V)$, and choose
$\delta>0$ with $2\delta<\mu(V)-\mu(U)$. Apply balanced
towers to $\{1_U,1_V\}$, with minimum height $1$.
Refine each base $B$ by the clopen sets $B\cap s^{-1}U$
and $B\cap s^{-1}V$, for all $s$ in its shape. Keep the
same shape above each refined base. This remains a tower
partition and retains the balancing inequalities, while every
level now lies either wholly in or wholly outside each of $U,V$.

For a resulting tower $(B',S)$, let
\[
 S_U=\{s\in S:sB'\subseteq U\},\qquad
 S_V=\{s\in S:sB'\subseteq V\}.
\]
The balancing inequalities give
\[
 \frac{|S_U|}{|S|}<\mu(U)+\delta
 <\mu(V)-\delta<\frac{|S_V|}{|S|}.
\]
Choose an injection $\iota:S_U\to S_V$, and on each level
$sB'\subseteq U$ use the group element $\iota(s)s^{-1}$ to
map it onto $\iota(s)B'\subseteq V$.
The source levels, over all towers, form a finite clopen partition
of $U$. Their images are disjoint within each tower because
$\iota$ is injective, and are disjoint between towers because
the tower levels partition $X$. Thus these translations define
a subequivalence $U\prec_G V$, which proves dynamical comparison.
\end{proof}

\section{Norm averaging in reduced crossed products}\label{S:average}

\subsection{Averaging in matrices}
The basic operation is a change of basis that replaces diagonal
entries by their average. For example, with
\[
 H=\frac1{\sqrt2}\begin{pmatrix}1&1\\1&-1\end{pmatrix},
 \qquad
 H^*\begin{pmatrix}a&0\\0&b\end{pmatrix}H
 =\frac12\begin{pmatrix}a+b&a-b\\a-b&a+b\end{pmatrix},
\]
both new diagonal entries equal $(a+b)/2$.
More generally, put $\omega_n=e^{2\pi i/n}$. The \emph{Fourier
matrix} is the unitary
\begin{equation}\label{E:Fourier-matrix}
 F_n=\bigl[n^{-1/2}\omega_n^{pq}\bigr]_{p,q=0}^{n-1}.
\end{equation}
Orthogonality of the powers of $\omega_n$ gives $F_n^*F_n=1$ and
\[
 (F_n^*\diag(a_0,\ldots,a_{n-1})F_n)_{jj}
 =\frac1n\sum_{p=0}^{n-1}a_p.
\]
Thus Fourier conjugation averages every diagonal matrix exactly.

For a general matrix $X$, its off-diagonal entries also contribute
to the new diagonal. We first insert phases which temper these contributions. Indeed, if
$U_\zeta=\diag(\zeta_0,\ldots,\zeta_{n-1})$ with $|\zeta_p|=1$,
then
\begin{equation}\label{E:phase-averaging}
 \bigl(F_n^*U_\zeta^*XU_\zeta F_n\bigr)_{jj}
 -\tr_n(X)
 =\frac1n\sum_{p\ne q}\overline\zeta_p\zeta_q
                 \omega_n^{(q-p)j}X_{pq},
\end{equation}
where $\tr_n=n^{-1}\operatorname{Tr}$.
The phases leave the original diagonal unchanged and allow the
off-diagonal contributions to cancel (for suitable choices of $U_{\zeta}$).

The next lemma gives one unitary that works simultaneously for
finitely many matrices. Its proof has two steps. Averaging over
finitely many choices of phases makes all but a bounded number of
diagonal positions good. Each exceptional position is then mixed,
using another Fourier matrix, with many good positions that have
small matrix entries between them. This dilutes the exceptional
value while keeping the off-diagonal error small.

\begin{lemma}\label{L:matrix}
For every $k\in\Nb$ and $M,\eta>0$, there exists $N\in\Nb$ such
that, whenever $n\ge N$ and $X_1,\ldots,X_k\in M_n(\Cb)$ have
norm at most $M$, there is $W\in\cU(M_n(\Cb))$ with
\[
 \max_{1\le r\le k,\,0\le j<n}
 |(W^*X_rW)_{jj}-\tr_n(X_r)|<\eta.
\]
\end{lemma}

\begin{proof}
Put $\delta=\eta/3$. Let $X_1,\ldots,X_k\in M_n(\mathbb C)$
have norm at most $M$, and write $\tau_r=\operatorname{tr}_n(X_r)$.
We shall construct the required unitary whenever $n$ exceeds a bound
depending only on $k,M,\eta$.

\medskip
\noindent
\textit{Step 1: make most diagonal positions good.}
Let $\Xi=\{1,-1,i,-i\}$. An element $\zeta\in\Xi^n$ is a
phase vector $\zeta=(\zeta_0,\ldots,\zeta_{n-1})$. Define the
averaging operator $\operatorname{Av}$ on complex-valued functions
on $\Xi^n$ by
\[
\operatorname{Av}(h)
:=
\frac{1}{4^n}\sum_{\zeta\in\Xi^n}h(\zeta).
\]
Thus $\operatorname{Av}(h)$ is the arithmetic mean of the values
of $h$ over all $4^n$ phase vectors. For distinct indices $p,q$, define
\[
\chi_{pq}(\zeta)=\overline{\zeta_p}\zeta_q.
\]
These functions satisfy
\[
\operatorname{Av}(\chi_{pq}\overline{\chi_{ab}})
=
\begin{cases}
1, & (p,q)=(a,b),\\
0, & (p,q)\ne(a,b),
\end{cases}
\qquad p\ne q,\ a\ne b.
\]
Indeed, in the first case the product is identically $1$.
Otherwise, some coordinate of $\zeta$ occurs with a nonzero
exponent in $\{-2,-1,1,2\}$. Averaging over that coordinate
gives zero, since
\[
\frac14\sum_{\xi\in\Xi}\xi^\ell=0, 
\quad \ell\in\{-2,-1,1,2\}.
\]
For each $\zeta\in\Xi^n$, put
\[
U_\zeta=\operatorname{diag}(\zeta_0,\ldots,\zeta_{n-1}),
\qquad
Y_r(\zeta)=F_n^*U_\zeta^*X_rU_\zeta F_n.
\]
Write $Y_{r,jj}$ for the scalar-valued function
$\zeta\mapsto Y_r(\zeta)_{jj}$, and identify scalars with
constant functions on $\Xi^n$. Writing $\omega_n=e^{2\pi i/n}$,
the Fourier-conjugation formula gives the identity of functions
\[
Y_{r,jj}-\tau_r
=
\frac1n\sum_{p\ne q}
\omega_n^{(q-p)j}(X_r)_{pq}\chi_{pq}.
\]
The orthogonality established above eliminates all cross terms
when we average the squared absolute value. Hence, for each $r,j$,
\[
\operatorname{Av}(|Y_{r,jj}-\tau_r|^2)
=
\frac1{n^2}\sum_{p\ne q}|(X_r)_{pq}|^2.
\]
By linearity of $\operatorname{Av}$,
\begin{align*}
\operatorname{Av}\Big(\sum_{r=1}^k\sum_{j=0}^{n-1}|Y_{r,jj}-\tau_r|^2\Big)
&=\frac1n\sum_{r=1}^k\sum_{p\ne q}|(X_r)_{pq}|^2\\
&\le\frac1n\sum_{r=1}^k\operatorname{Tr}(X_r^*X_r) \le kM^2.
\end{align*}
Since this is a finite average, there is a phase vector $\zeta$
for which the expression being averaged is at most $kM^2$.
Fix such a vector, and put
\[
W_0=U_\zeta F_n,
\qquad
Y_r=W_0^*X_rW_0.
\]
Thus the matrices $Y_r$ satisfy
\[
\sum_{r=1}^k\sum_{j=0}^{n-1}|Y_{r,jj}-\tau_r|^2
\le kM^2.
\]
Call a position $j$ \emph{bad} if
$|Y_{r,jj}-\tau_r|>\delta$ for some $r$, and \emph{good}
otherwise. Each bad position contributes more than $\delta^2$
to the preceding sum, so there are at most
\[
B=\Bigg\lceil\frac{kM^2}{\delta^2}\Bigg\rceil
\]
bad positions. In particular, their number is bounded
independently of $n$.

\medskip
\noindent
\textit{Step 2: spread each bad position among good ones.}
Choose an integer $m\ge2$ such that $2M/m<\delta$, and put
$\theta=\delta/m$. Since $\|Y_r\|\le M$, the row and column
norm estimates give, for each index $p$,
\[
\sum_{q=0}^{n-1}\sum_{r=1}^k
\bigl(|Y_{r,pq}|^2+|Y_{r,qp}|^2\bigr)
\le 2kM^2.
\]
Consequently, at most
\[
D=\Bigg\lceil\frac{2kM^2}{\theta^2}\Bigg\rceil
\]
indices $q\ne p$ satisfy $|Y_{r,pq}|>\theta$ or
$|Y_{r,qp}|>\theta$ for some $r$. Set
\[
N=m(B+D)+1,
\]
and suppose that $n\ge N$. Enumerate the bad positions as $b_1,\ldots,b_t$, where $t\le B$.
We claim that there are pairwise disjoint $m$-element index sets
\[
I_1,\ldots,I_t\subseteq\{0,\ldots,n-1\}
\]
such that $I_j$ contains $b_j$ and $m-1$ good positions, and, for distinct $p, q\in I_j$ and $1\le r\le k$, 
\[
|Y_{r,pq}|\le\theta. 
\]
The indices in each $I_j$ need not be consecutive. Construct these sets successively, starting $I_j$ with $b_j$ and adding good indices one at a time. At each addition,
exclude all bad indices and all good indices already selected,
including those in the current set. These exclude at most
\[
B+(m-1)B=mB
\]
choices, since there are at most $B$ sets to construct, each
requiring $m-1$ good indices. Also exclude every index $q$
for which $|Y_{r,pq}|>\theta$ or $|Y_{r,qp}|>\theta$ for some
$p$ in the current set and some $r$. Each such $p$ excludes
at most $D$ choices, and the current set has fewer than $m$
elements, so these exclusions account for at most $mD$
further choices. Thus at most $m(B+D)<n$ indices are excluded.
An unused good index satisfying all the required inequalities
therefore remains available. Adding it preserves the
inequalities between every pair of distinct indices in the
current set, so this procedure completes the construction.

Let $V$ act as $F_m$ on each coordinate subspace $\operatorname{span}\{e_p:p\in I_j\}$, 
with its coordinates listed in increasing order, and as the
identity on the remaining coordinates. The disjointness of
the index sets makes $V$ unitary.

Fix $r$ and one of the sets $I_j$. Fourier conjugation makes
the diagonal part of the compression of $Y_r$ contribute
\[
\frac1m\sum_{p\in I_j}Y_{r,pp}
\]
to each new diagonal entry on this set. Since $I_j$ contains
only one bad position, and
\[
|Y_{r,pp}-\tau_r|
\le |Y_{r,pp}|+|\tau_r|
\le 2M
\]
at every position, we have
\[
\Bigg|
\frac1m\sum_{p\in I_j}Y_{r,pp}-\tau_r
\Bigg|
\le
\frac{2M+(m-1)\delta}{m}
<2\delta.
\]
Each entry of $F_m$ has modulus $m^{-1/2}$, so the
off-diagonal contribution to any new diagonal entry has
absolute value at most
\[
\frac1m
\sum_{\substack{p,q\in I_j\\p\ne q}}
|Y_{r,pq}|
\le (m-1)\theta
<\delta.
\]
Every diagonal entry of $V^*Y_rV$ on the selected index sets
is therefore within $3\delta=\eta$ of $\tau_r$. All remaining
positions were already good, and their diagonal entries
are unchanged. It follows that $W=W_0V$ satisfies the conditions of the lemma.
\end{proof}

\subsection{Tower matrix algebras}

Fix a clopen tower $(B,S)$, and enumerate
$S=\{s_1,\ldots,s_n\}$. Functions on $B$ are extended by zero
to $X$. The map
\begin{equation}\label{E:tower-algebra}
\Phi_B:M_n(C(B))\rightarrow C(X)\rtimes_\lambda G,
\quad
[g_{ij}]\mapsto
\sum_{i,j=1}^n u_{s_i}g_{ij}u_{s_j}^*
\end{equation}
is an injective $*$-homomorphism whose image has unit $p_B=1_{\bigsqcup_{s\in S}sB}$. The conditional expectation $E$ recovers the diagonal entries. Indeed for a function-valued matrix $[g_{ij}]\in M_n(C(B))$, $y\in B$, and $1\le j\le n$, we have 
\[
E(\Phi_B([g_{ij}]))(s_jy)=g_{jj}(y).
\]
For $f\in C(X)$, the element $p_Bfp_B$ corresponds under
$\Phi_B$ to the diagonal matrix-valued function
\[
y\mapsto
\diag(f(s_1y),\ldots,f(s_ny)).
\]
Viewing the Fourier matrix $F_n$ from
\eqref{E:Fourier-matrix} as a constant matrix over $B$,
we therefore have
\[
E\bigl(\Phi_B(F_n)^*f\Phi_B(F_n)\bigr)(s_jy)
=
\frac1n\sum_{s\in S}f(sy).
\]
Thus Fourier conjugation replaces the values of $f$ along
each tower fibre by their average. The balancing condition
makes this average uniformly close to $\mu(f)$.
Figure~\ref{fig:tower-averaging} illustrates the operation
on a tower of height four.

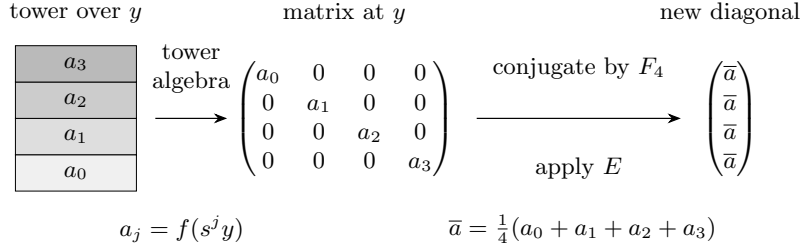
\begin{figure}[htbp]
\centering
\begin{tikzpicture}[x=1cm,y=1cm,font=\footnotesize,>=Stealth]
  \node at (0.8,2.35) {tower over $y$};
  \node at (4.35,2.35) {matrix at $y$};
  \node at (9.45,2.35) {new diagonal};

  \foreach \j/\shade in {0/6,1/13,2/20,3/27} {
    \pgfmathsetmacro{\bottom}{0.48*\j}
    \pgfmathsetmacro{\top}{0.48*(\j+1)}
    \pgfmathsetmacro{\mid}{0.48*(\j+0.5)}
    \filldraw[fill=black!\shade,draw=black,line width=0.35pt]
      (0,\bottom) rectangle (1.6,\top);
    \node at (0.8,\mid) {$a_{\j}$};
  }

  \node at (4.35,0.96) {$\begin{pmatrix}
    a_0&0&0&0\\
    0&a_1&0&0\\
    0&0&a_2&0\\
    0&0&0&a_3
  \end{pmatrix}$};

  \node at (9.45,0.96) {$\begin{pmatrix}
    \overline a\\
    \overline a\\
    \overline a\\
    \overline a
  \end{pmatrix}$};

  \draw[->,line width=0.45pt] (1.85,0.96) -- (2.8,0.96);
  \node[align=center] at (2.325,1.62) {tower\\algebra};

  \draw[->,line width=0.45pt] (6.1,0.96) -- (8.8,0.96);
  \node at (7.45,1.62) {conjugate by $F_4$};
  \node at (7.45,0.3) {apply $E$};

  \node at (2.2,-0.5) {$a_j=f(s^jy)$};
  \node at (7.5,-0.5)
    {$\overline a=\tfrac14(a_0+a_1+a_2+a_3)$};
\end{tikzpicture}
\caption{Averaging a coefficient on a single tower. For
illustration, the shape is the interval $\{e,s,s^2,s^3\}$.}
\label{fig:tower-averaging}
\end{figure}

For general Fourier polynomials, the corresponding matrices
may have nonzero off-diagonal entries. The following
compression formula allows us to apply Lemma~\ref{L:matrix}
to these matrices and recover the resulting diagonal
estimates through $E$. Let $x=\sum_{g\in F}f_gu_g$ be a Fourier polynomial, where
$F\subseteq G$ is finite, and set $f_g=0$ for $g\notin F$. For $y\in B$, define
\begin{equation}\label{E:compression-matrix}
X_x(y)=\bigl[f_{s_i s_j^{-1}}(s_i y)\bigr]_{i,j=1}^n.
\end{equation}

\begin{lemma}\label{L:tower-matrices}
The map $y\mapsto X_x(y)$ is continuous, and
$\|X_x(y)\|\le\|x\|$ for every $y\in B$. Moreover,
\begin{equation}\label{E:matrix-trace}
\tr_n(X_x(y))
=
\frac1n\sum_{s\in S}E(x)(sy).
\end{equation}
If $W\in\cU(n)$ is viewed as a constant matrix over $B$, then
for every $y\in B$ and $1\le j\le n$,
\begin{equation}\label{E:expectation-diagonal}
E\bigl(\Phi_B(W)^*x\Phi_B(W)\bigr)(s_jy)
=
(W^*X_x(y)W)_{jj}.
\end{equation}
\end{lemma}

\begin{proof}
Continuity follows from the formula for the matrix entries.
For $y\in B$, consider the regular representation $\pi_y$
on $\ell^2(G)$, given by
\[
\pi_y(f)\delta_s=f(sy)\delta_s,
\qquad
\pi_y(u_g)\delta_s=\delta_{gs}.
\]
With respect to the ordered basis
$\delta_{s_1},\ldots,\delta_{s_n}$, the matrix $X_x(y)$ is
the compression of $\pi_y(x)$ to $\ell^2(S)$. Hence
\[
\|X_x(y)\|\le\|\pi_y(x)\|\le\|x\|.
\]
Its diagonal entries are $f_e(s_jy)=E(x)(s_jy)$, which
gives \eqref{E:matrix-trace}. Put $w=\Phi_B(W)$. Since the tower levels are pairwise
disjoint and $y\in B$, the definition of $\Phi_B$ gives
\[
\pi_y(w)\delta_{s_j}
=
\sum_{i=1}^n W_{ij}\delta_{s_i}.
\]
Applying the same observation to $w^*=\Phi_B(W^*)$ shows
that $\ell^2(S)$ is invariant under both $\pi_y(w)$ and
$\pi_y(w^*)$, with restrictions represented by $W$ and
$W^*$, respectively. The compression of $\pi_y(w^*xw)$
to $\ell^2(S)$ therefore has matrix $W^*X_x(y)W$. For any $a\in C(X)\rtimes_\lambda G$, the diagonal matrix coefficient of $\pi_y(a)$ at $\delta_{s_j}$ is $E(a)(s_jy)$.
Taking the $j$th diagonal entry of the preceding compression
thus gives \eqref{E:expectation-diagonal}.
\end{proof}

\begin{prop}\label{P:averaging}
Let $G\curvearrowright X$ be an action of a countable discrete group on the Cantor set. Suppose that the action preserves a probability measure $\mu$ of full support and has balanced towers with respect
to $\mu$. Then the canonical expectation $E$ on $C(X)\rtimes_\lambda G$ has norm averaging with respect to $\tau=\mu\circ E$. The averaging unitaries can be chosen in the algebraic crossed product.
\end{prop}

\begin{proof}
Let $A = C(X)\rtimes_\lambda G$, and let $F\subset A$ be a nonempty finite set of Fourier
polynomials, and fix $\varepsilon>0$. Put $\eta=\varepsilon/4$
and choose $M>\max_{x\in F}\|x\|$. Apply Lemma~\ref{L:matrix}
with $k=|F|$, norm bound $M$, and tolerance $\eta$, obtaining
a minimum matrix size $N$.

Choose a balanced tower partition $\{(B_i,S_i)\}_{i=1}^l$ for
the functions in the set $\{E(x):x\in F\}$, tolerance $\eta$, and minimum
height $N$. Enumerate $S_i=\{s_{i,1},\ldots,s_{i,n_i}\}$, and
write $X_{i,x}(y)$ for the matrix in
\eqref{E:compression-matrix} on this tower. Thus, for $y\in B_i$
and $x\in F$,
\begin{equation}\label{E:balanced-matrix-traces}
 \|X_{i,x}(y)\|\le\|x\|<M,
 \qquad
 |\tr_{n_i}(X_{i,x}(y))-\tau(x)|<\eta.
\end{equation}
The second inequality is precisely the balancing condition,
by \eqref{E:matrix-trace}. It remains to make every diagonal
entry close to this normalised trace, using unitaries that vary
locally constantly over the bases.

Fix $i$. For each $y_0\in B_i$, Lemma~\ref{L:matrix} provides
$W_{i,y_0}\in\cU(n_i)$ such that, for all $x\in F$, $1\le j\le n_i$, 
\[
 |(W_{i,y_0}^*X_{i,x}(y_0)W_{i,y_0})_{jj}
             -\tr_{n_i}(X_{i,x}(y_0))|<\eta.
\]
Each matrix field $X_{i,x}$ is continuous in norm, since its finitely
many entries are continuous functions of $y$. Since $F$ is finite
and $X$ has a clopen basis, there is a clopen neighborhood
$O_{i,y_0}\subseteq B_i$ of $y_0$ such that, for $y\in O_{i,y_0}, x\in F$,
\[
 \|X_{i,x}(y)-X_{i,x}(y_0)\|<\eta/2.
\]
The normalised trace and each diagonal-entry functional have
norm one. Hence, for every $y\in O_{i,y_0}$,
\begin{align*}
 &|(W_{i,y_0}^*X_{i,x}(y)W_{i,y_0})_{jj}
                -\tr_{n_i}(X_{i,x}(y))|\\*
 &\quad\le
 |(W_{i,y_0}^*X_{i,x}(y_0)W_{i,y_0})_{jj}
                -\tr_{n_i}(X_{i,x}(y_0))|
       +2\|X_{i,x}(y)-X_{i,x}(y_0)\|<2\eta.
\end{align*}
By compactness, finitely many of the sets $O_{i,y_0}$ cover
$B_i$. Taking successive differences gives a finite clopen
partition $B_i=\bigsqcup_a C_{i,a}$ subordinate to this cover;
discard empty parts. On each $C_{i,a}$ choose the constant
unitary $W_{i,a}$ associated to one covering neighborhood.
For every $y\in C_{i,a}$, $x\in F$, and $1\le j\le n_i$, we have
\begin{equation}\label{E:local-matrix-averaging}
 |(W_{i,a}^*X_{i,x}(y)W_{i,a})_{jj}
                -\tr_{n_i}(X_{i,x}(y))|<2\eta.
\end{equation}
The pairs $(C_{i,a},S_i)$ form a tower partition with the same heights
and balancing estimates as before. Thus no continuous choice
of the pointwise unitaries is required.

We now assemble the locally constant choices. For each refined
tower put
\[
 p_{i,a}=\sum_{j=1}^{n_i}1_{s_{i,j}C_{i,a}},\qquad
 w_{i,a}=\sum_{p,q=1}^{n_i}(W_{i,a})_{pq}
               u_{s_{i,p}}1_{C_{i,a}}u_{s_{i,q}}^*.
\]
The tower-algebra embedding \eqref{E:tower-algebra} gives
\[
 w_{i,a}=p_{i,a}w_{i,a}p_{i,a},\qquad
 w_{i,a}^*w_{i,a}=w_{i,a}w_{i,a}^*=p_{i,a}.
\]
The projections $p_{i,a}$ are pairwise orthogonal and sum to
$1$, so $w=\sum_{i,a}w_{i,a}$ satisfies $w^*w=ww^*=1$.
Moreover, every summand in its defining formula is a Fourier
monomial, since
\[
 u_{s_{i,p}}1_{C_{i,a}}u_{s_{i,q}}^*
     =1_{s_{i,p}C_{i,a}}u_{s_{i,p}s_{i,q}^{-1}}.
\]
Thus $w$ belongs to the algebraic crossed product.

Although $x$ can couple distinct towers, their cross terms
vanish after applying $E$. Indeed, for $(i,a)\ne(h,b)$,
$C(X)$-bimodularity gives
\[
 E(w_{i,a}^*xw_{h,b})
   =p_{i,a}E(w_{i,a}^*xw_{h,b})p_{h,b}=0.
\]
Each remaining term is supported on its own tower. Consequently,
\eqref{E:expectation-diagonal} gives, for $y\in C_{i,a}$,
\[
 E(w^*xw)(s_{i,j}y)
       =(W_{i,a}^*X_{i,x}(y)W_{i,a})_{jj}.
\]
Combining \eqref{E:balanced-matrix-traces} and
\eqref{E:local-matrix-averaging}, we obtain
\begin{align*}
 |E(w^*xw)(s_{i,j}y)-\tau(x)|
 &\le |(W_{i,a}^*X_{i,x}(y)W_{i,a})_{jj}
                       -\tr_{n_i}(X_{i,x}(y))|\\*
 &\quad+|\tr_{n_i}(X_{i,x}(y))-\tau(x)|<3\eta.
\end{align*}
The levels cover $X$, so
$\|E(w^*xw)-\tau(x)1\|\le3\eta<\varepsilon$ for every
$x\in F$.

Finally, let $F\subset A$ be any finite set. For each $x\in F$
choose a Fourier polynomial $a_x$ with
$\|x-a_x\|<\varepsilon/4$, and apply the preceding argument
to $\{a_x:x\in F\}$ with tolerance $\varepsilon/2$. The resulting
algebraic unitary $w$ satisfies
\[
 \|E(w^*xw)-\tau(x)1\|
 \le2\|x-a_x\|+\|E(w^*a_xw)-\tau(a_x)1\|<\varepsilon,
\]
using contractivity of $E$ and $\tau$. This proves norm averaging
with algebraic averaging unitaries.
\end{proof}

\begin{cor}\label{C:comparison-NA}
Let $G\curvearrowright X$ be a strictly ergodic action of a countable discrete group on the Cantor set with dynamical comparison. Then the canonical conditional expectation on $C(X)\rtimes_\lambda G$ has norm averaging with respect to $\tau = \mu \circ E$, where $\mu$ is the unique invariant Borel probability measure.
\end{cor}

\begin{proof}
Combine Propositions~\ref{P:comparison} and \ref{P:averaging}. 
\end{proof}

\section{Norm averaging and PHP}\label{S:group-PHP}

We now combine norm averaging with Ozawa's PHP construction to obtain complete selflessness. The operator estimates used below come from the proof of \cite[Theorem~14]{Oza25}. Their
crossed product form appears explicitly in the proof of \cite[Theorem~1.4]{FloKliCobPag26b}. Our use of norm averaging separates the control of the identity Fourier coefficient from the
choice of PHP data. We first conjugate by an averaging unitary, apply the PHP construction to the remaining coefficients, and then conjugate back.

\begin{thm}\label{T:group-PHP-NA}
Let $G$ be a countable group with property PHP and $G\curvearrowright D$ an action on a separable unital C$^*$-algebra $D$ with an invariant faithful tracial state $\sigma$. If the canonical expectation $E$ on $A=D\rtimes_\lambda G$ has norm averaging with respect to $\tau=\sigma\circ E$, then $(A,\tau)$ is completely selfless.
\end{thm}

\begin{proof}
Represent $A$ regularly on $H=\ell^2(G)\otimes L^2(D,\sigma)$ and put $\Omega=\delta_e\otimes\widehat{1}$. The vector state $\psi$ on $B(H)$ defined by $\psi(z)=\langle z\Omega,\Omega\rangle$ restricts to $\tau$ on $A$. For a finite subset $K\subseteq G$, let $R_K$ denote the projection onto $\ell^2(K)\otimes L^2(D,\sigma)$. Fix finite subsets $F,L$ of the unit ball of $A$ and $\varepsilon>0$. By norm averaging, choose $w\in\cU(A)$
such that, for all $x\in F$,
\[
\|E(w^*xw)-\tau(x)1\|<\varepsilon.
\]
For each $x\in F$, choose a Fourier polynomial $y_x$ with $\|y_x-w^*xw\|<\varepsilon$. 
Then $\|E(y_x)-\tau(x)1\|<2\varepsilon$. Let $S\subseteq G\setminus\{e\}$ be the union of the
nonidentity Fourier supports of the $y_x$. Also choose a finite subset $K\subseteq G$ such that, for all $a\in L$,
\[
\|(1-R_K)w^*a^*\Omega\|<\varepsilon.
\]
We apply the PHP operator construction from the proofs of \cite[Theorem~14]{Oza25} and \cite[Theorem~1.4]{FloKliCobPag26b} to the finite set $\{e\}\cup S$. After choosing the PHP
tolerance sufficiently small, the estimates established there give, for arbitrarily large integers $n$, a contraction $T\in B(H)$ satisfying
\begin{enumerate}
	\item $\|T^*T-1\|<\varepsilon$;
	\item $\operatorname{dist}(T+T^*,A)<\varepsilon$;
	\item $T^*du_gT=0$ for all $d\in D, g\in S$; and 
	\item $\|R_KT\|^2\le |K|/(2n)$. 
\end{enumerate}
Only the group PHP conditions enter these estimates. The coefficient averaging condition in \cite{FloKliCobPag26b} is not needed here. Choose $n$ large enough that $\sqrt{|K|/(2n)}<\varepsilon$, and put $V=wTw^*$. Then $V$ is a contraction and
\[
\|V^*V-1\|<\varepsilon,
\qquad
\operatorname{dist}(V+V^*,A)<\varepsilon.
\]
For each $x\in F$, the annihilation of the nonidentity Fourier terms gives $T^*y_xT=T^*E(y_x)T$. Hence, using $\|T\|\le1$ and $|\tau(x)|\le1$, we have 
\begin{align*}
\|V^*xV-\tau(x)1\|
&=\|T^*w^*xwT-\tau(x)1\|\\
&\le
\|w^*xw-y_x\|
+\|T^*(E(y_x)-\tau(x)1)T\|\\
&\qquad
+|\tau(x)|\,\|T^*T-1\| <4\varepsilon.
\end{align*}
The choice of $K$ controls the state after conjugating by $w$. For every $a\in L$,
\begin{align*}
\psi(aVV^*a^*)^{1/2}
&=\|V^*a^*\Omega\|\\
&=\|T^*w^*a^*\Omega\|\\
&\le
\|T^*R_K\|\,\|w^*a^*\Omega\|
+\|(1-R_K)w^*a^*\Omega\|\\
&\le
\sqrt{|K|/(2n)}+\varepsilon < 2\varepsilon.
\end{align*}
By separability, apply this construction to increasing finite sets exhausting a countable dense subset of the unit ball of $A$, with $\varepsilon\to0$. The resulting contractions define an isometry $V\in B(H)^{\mathcal U}$ for a free ultrafilter $\mathcal U$ on $\mathbb N$ satisfying the assumptions of Ozawa's criterion Theorem~\ref{T:Ozawa-isometry}. Thus $(A,\tau)$ is completely selfless.
\end{proof}

\section{Selfless crossed products and regularity}\label{S:applications}

\subsection{Real rank zero}

In order to prove real rank zero, we use the following observation, an application of R{\o}rdam's criterion. The author thanks Hannes Thiel for pointing out this argument.

\begin{prop}\label{P:rr0}
Let $A$ be a simple unital C$^*$-algebra with stable rank one
and strict comparison. Suppose that $A$ has a tracial state
$\tau$ which is its only quasitracial state. If $A$
contains a unital copy of $C(X)$ for a Cantor space $X$, then
$A$ has real rank zero.
\end{prop}

\begin{proof}
The trace $\tau$ is faithful by simplicity. Partition $X$ into
$n$ nonempty clopen sets. Their characteristic functions are
nonzero orthogonal projections summing to $1$, so one of them,
say $p_n$, satisfies $0<\tau(p_n)\le1/n$. Hence the subgroup
$\tau_*(K_0(A))\subseteq\Rb$ contains arbitrarily small positive
elements and is dense. Since $T(A)=\{\tau\}$, this says that
the range of $K_0(A)\to\operatorname{Aff}(T(A))$ is dense. Strict comparison is equivalent to almost unperforation of $W(A)$ in this setting; see \cite[Section 4]{Ror04}.
The proof of R{\o}rdam's real rank zero criterion
\cite[Proposition 7.1]{Ror04} therefore applies. Its exactness
hypothesis is used only to identify quasitraces with traces.
\end{proof}

\subsection{The main results}

\begin{proof}[Proof of Theorem~\ref{T:PHP-main}]
Corollary~\ref{C:comparison-NA} gives norm averaging, and Theorem~\ref{T:group-PHP-NA} gives complete
selflessness. By \cite[Theorem 3.1]{Rob25}, the crossed product has
stable rank one and strict comparison, and the canonical trace is
its only quasitracial state. Proposition~\ref{P:rr0} gives real
rank zero.
\end{proof}

\begin{proof}[Proof of Corollary~\ref{T:main}]
The action of a nonamenable free product on the boundary of its Bass--Serre tree is minimal, extremely proximal, and topologically free, so \cite[Proposition 15]{Oza25} gives PHP. Applying Theorem~\ref{T:PHP-main} proves the corollary.
\end{proof}

\begin{proof}[Proof of Corollaries~\ref{T:amenable} and~\ref{T:free}]
For Corollary~\ref{T:amenable}, the restricted action has comparison
by \cite[Theorem A]{GlaLiu26}. Proposition~\ref{P:subgroup-comparison}
gives strict ergodicity and comparison for the $G$-action, so
Theorem~\ref{T:PHP-main} applies.

For Corollary~\ref{T:free}, use the cyclic subgroup generated by
the specified element, which has comparison by \cite{Mat12,Ker20}. Apply
Proposition~\ref{P:subgroup-comparison} and Theorem~\ref{T:PHP-main},
using PHP for $F_d$. Every action in $\Astar(F_d,X)$ has a
strictly ergodic generator by definition, giving the last assertion.
\end{proof}

\begin{proof}[Proof of Corollary~\ref{C:BD}]
	 The Bunce--Deddens algebra $\mathsf{BD}(G,\sigma)$ is by definition the crossed product of the free minimal odometer action $G\curvearrowright \widehat{G}_\sigma$ (see \cite{Bel26}), which is known to have comparison \cite[Lemma~2.4]{Jos24}. This action is also strictly ergodic \cite{CorPet08,Bel26} and so Theorem~\ref{T:PHP-main} applies.
\end{proof}

\begin{rem}\label{R:WA}
In contrast to $\Astar(F_d,X)$, the standard generators of a generic action in $\operatorname{WA}(F_d,X)$, consisting of topologically free minimal weakly mixing actions admitting an invariant Borel probability measure, are not minimal \cite[Proposition 8.6]{BelGefKer25}. Thus the argument used for $\Astar(F_d,X)$ does not obviously extend to $\operatorname{WA}(F_d,X)$. 
\end{rem}

\begin{rem}\label{R:comp}
Recently, Boldrini and Prasad constructed the first examples of minimal topologically free actions without dynamical comparison \cite{BolPra26}. These are actions of $F_\infty$ on the Cantor set which can be chosen to preserve no measure or at least one measure. We do not know whether their examples can be chosen to be strictly ergodic, though suspect that this is possible.
\end{rem}

In light of Remark~\ref{R:comp}, it seems natural to probe further the relationship between dynamical comparison and strict comparison/selflessness. The following asks for a strengthening of \cite[Corollary 7.5]{BolPra26}. 

\begin{question}
	Is there a strictly ergodic action $G\curvearrowright X$ on the Cantor set which fails to have dynamical comparison, but for which $C(X)\rtimes_\lambda G$ is (completely) selfless?
\end{question}

On the other hand, it is also possible that among finitely generated groups, comparison always holds for strictly ergodic Cantor actions. A related question is whether comparison holds for actions in $\operatorname{WA}(F_d,X)$ (cf.~Remark~\ref{R:WA}).

\begin{question}
	Is there a strictly ergodic action of a (nonamenable) finitely generated group on the Cantor set which fails to have dynamical comparison?
\end{question}

\section{Regular Toeplitz Cantor actions}\label{S:Toeplitz}

We now apply Theorem~\ref{T:PHP-main} to regular Toeplitz actions
and prove Corollary~\ref{C:Toeplitz}. After recalling the relevant
terminology, we show that strict ergodicity and comparison pass
to regular almost one-to-one extensions. This gives examples
beyond the equicontinuous setting, including actions for which
no amenable subgroup acts minimally.

\subsection{Subodometers and Toeplitz extensions}

Let $G$ be a countable discrete group and let $(G_n)_{n\geq1}$
be a decreasing sequence of finite-index subgroups. The associated
\emph{$G$-subodometer} is the action by coordinatewise left
multiplication on
\[
Y=\varprojlim G/G_n,
\]
where the connecting maps are $gG_{n+1}\mapsto gG_n$. When all
the subgroups $G_n$ are normal, we call this a \emph{$G$-odometer}.
These actions are strictly ergodic, with invariant measure
projecting to the uniform probability measure on each finite
quotient; see \cite{CorPet08,Bel26}. The space $Y$ is a Cantor
space when $[G:G_n]\to\infty$. Every $G$-subodometer has dynamical
comparison by \cite[Lemma~2.4]{Jos24}.

A factor map $\pi:X\to Y$ between minimal compact metrisable
$G$-systems is \emph{almost one-to-one} if
\[
Y_1=\{y\in Y:|\pi^{-1}(y)|=1\}
\]
is residual in $Y$. If $Y$ is uniquely ergodic, with invariant
probability measure $\nu$, we call the extension \emph{regular}
when $\nu(Y_1)=1$.

For a finite alphabet $\mathcal A$, equip $\mathcal A^G$ with
the shift action
\[
(g\cdot x)(h)=x(g^{-1}h).
\]
A configuration $x\in\mathcal A^G$ is \emph{Toeplitz} if, for
every $h\in G$, there is a finite-index subgroup $H\leq G$ such
that $x(kh)=x(h)$ for all $k\in H$. Its orbit closure is called
a \emph{Toeplitz $G$-subshift}. These are precisely the minimal
symbolic almost one-to-one extensions of $G$-subodometers;
see \cite{CorPet08,CecCorGom24}.

The subodometer in this characterisation is the maximal
equicontinuous factor. A Toeplitz subshift is \emph{regular}
when its canonical factor map to this subodometer is regular
in the sense above \cite[Definition~3.4]{CecCorGom24}. We use the term \emph{regular Toeplitz Cantor action} for an action conjugate to an infinite regular Toeplitz subshift.

\subsection{Comparison and applications}

The following observation transfers comparison from a subodometer
to its regular Toeplitz extensions. 

\begin{lemma}\label{L:factor}
Let $G\curvearrowright Y$ be a strictly ergodic action on the
Cantor set with dynamical comparison, and let $\nu$ be its
invariant Borel probability measure. Suppose that
$\pi:X\to Y$ is a factor map from a minimal Cantor action
$G\curvearrowright X$ and that
\[
\nu\bigl(\{y\in Y:|\pi^{-1}(y)|=1\}\bigr)=1.
\]
Then $G\curvearrowright X$ is strictly ergodic and has
dynamical comparison.
\end{lemma}

\begin{proof}
Write $Y_1=\{y\in Y:|\pi^{-1}(y)|=1\}$. This set is
$G$-invariant and Borel. Indeed, for any compatible metric on $X$,
\[
Y_1=\bigcap_{n\geq1}
\{y\in Y:\operatorname{diam}(\pi^{-1}(y))<1/n\},
\]
and the sets in this intersection are open. Let $s:Y_1\to X$
assign to each $y$ the unique point of $\pi^{-1}(y)$.
Compactness of $X$ implies that $s$ is continuous on $Y_1$,
and equivariance of $\pi$ implies that $s$ is equivariant.
Thus $\mu=s_*(\nu|_{Y_1})$ is an invariant probability measure
on $X$. Every invariant probability measure on $X$ pushes
forward to $\nu$ by unique ergodicity, and hence gives full
measure to $\pi^{-1}(Y_1)$, where $\pi$ is one-to-one. It must
therefore equal $\mu$. Since the action on $X$ is minimal by
assumption, this proves strict ergodicity.

We next prove comparison. For a clopen set $C\subseteq X$, put
\[
C^+=\pi(C)
=\{y\in Y:\pi^{-1}(y)\cap C\ne\emptyset\}
\]
and
\[
C^-=Y\setminus\pi(X\setminus C)
=\{y\in Y:\pi^{-1}(y)\subseteq C\}.
\]
The set $C^+$ is closed and $C^-$ is open, and
\begin{equation}\label{E:enclose}
\pi^{-1}(C^-)\subseteq C\subseteq\pi^{-1}(C^+).
\end{equation}
Every point of $C^+\setminus C^-$ has a fibre meeting both
$C$ and $X\setminus C$, so $C^+\setminus C^-$ is disjoint
from $Y_1$ and has $\nu$-measure zero. Since $\pi_*\mu=\nu$,
the inclusions in \eqref{E:enclose} imply
\[
\nu(C^-)=\mu(C)=\nu(C^+).
\]

Fix clopen sets $U,V\subseteq X$ with $\mu(U)<\mu(V)$, and
choose $\varepsilon>0$ such that
$2\varepsilon<\mu(V)-\mu(U)$. By regularity of $\nu$ and
zero-dimensionality of $Y$, there are clopen sets $A,B\subseteq Y$
such that $U^+\subseteq A$, $B\subseteq V^-$, and
\[
\nu(A)<\mu(U)+\varepsilon<\mu(V)-\varepsilon<\nu(B).
\]
In particular,
\[
U\subseteq\pi^{-1}(A),\qquad \pi^{-1}(B)\subseteq V.
\]
Comparison for $G\curvearrowright Y$ gives a finite clopen
partition $A=\bigsqcup_{j=1}^m A_j$ and elements $g_j\in G$
such that the sets $g_jA_j$ are pairwise disjoint subsets of $B$.
Set $U_j=U\cap\pi^{-1}(A_j)$. These sets partition $U$, and
equivariance gives
\[
g_jU_j\subseteq\pi^{-1}(g_jA_j)
\subseteq\pi^{-1}(B)\subseteq V.
\]
The images $g_jU_j$ are pairwise disjoint because the sets
$g_jA_j$ are pairwise disjoint. Discarding empty pieces gives
$U\prec_G V$.
\end{proof}

We can now apply the lemma to the subodometer factors described
above.

\begin{prop}\label{P:Toeplitz}
Every regular Toeplitz Cantor action is strictly ergodic and
has dynamical comparison. Moreover, every countably infinite residually finite group admits a free regular Toeplitz Cantor action.
\end{prop}

\begin{proof}
Let $G\curvearrowright X$ be a regular Toeplitz Cantor action,
and let $\pi:X\to Y$ be its canonical subodometer factor.
The system $Y$ is strictly ergodic and has comparison.
It is also infinite: if $Y$ were finite, regularity would
force every fibre of $\pi$ to be a singleton, making $X$
finite. Thus $Y$ is a Cantor space, and Lemma~\ref{L:factor}
gives the first assertion.

For the second assertion, let $G$ be countably infinite and
residually finite, and choose a decreasing sequence $(G_n)$
of finite-index normal subgroups with
$\bigcap_n G_n=\{e\}$. By
\cite[Proposition~3.8]{CecCorGom24}, there exists a free
regular Toeplitz subshift $X\subseteq\{0,1\}^G$ with odometer
factor $\varprojlim G/G_n$. This subshift is infinite, hence
a Cantor system, and the first assertion gives strict
ergodicity and comparison.
\end{proof}

\begin{proof}[Proof of Corollary~\ref{C:Toeplitz}]
Proposition~\ref{P:Toeplitz} gives strict ergodicity and
dynamical comparison. Since $G$ has property PHP,
Theorem~\ref{T:PHP-main} applies.
\end{proof}

\begin{rem}\label{R:new}
An infinite Toeplitz subshift is expansive and hence cannot
be equicontinuous. Moreover, the free examples in
Proposition~\ref{P:Toeplitz} can be chosen outside the scope
of Corollary~\ref{T:amenable}. To see this, let $G=F_d$,
where $d\geq2$, and choose a surjection
\[
q:F_d\twoheadrightarrow Q=(\mathbb Z/2\mathbb Z)^2.
\]
Choose the residual normal chain $(N_n)$ used in
Proposition~\ref{P:Toeplitz} so that $N_1=\ker q$.
The resulting regular Toeplitz action factors onto its
odometer and hence onto the left translation action on $Q$.
Every amenable subgroup $H\leq F_d$ is cyclic or trivial,
so $q(H)$ has at most two elements and cannot act
transitively on $Q$. Since minimality passes to factors,
no amenable subgroup acts minimally on the Toeplitz system.
Thus these are free, strictly ergodic Cantor actions with
comparison to which Theorem~\ref{T:PHP-main} applies but
Corollary~\ref{T:amenable} does not.
\end{rem}

\end{document}